\documentclass[11pt]{amsart}
\usepackage{amsmath,amssymb,amsfonts,amsthm,amscd,indentfirst}

\numberwithin{equation}{section}
\usepackage{xcolor}

\def\<{\langle}
\def\>{\rangle}

\newtheorem{theorem}{Theorem}[section]
\newtheorem{proposition}[theorem]{Proposition}
\newtheorem{lemma}[theorem]{Lemma}
\newtheorem{remark}[theorem]{Remark}

\newtheorem{definition}[theorem]{Definition}

\begin{document}
\title{A special concavity property for positive Hessian quotient operators}
\author{Pengfei Guan and Marcin Sroka}
\address{Department of Mathematics\\
         McGill University\\
         Montreal, Quebec. H3A 2K6, Canada.}
\email{pengfei.guan@mcgill.ca}
\address{Faculty of mathematics and computer science, Jagiellonian University, \L ojasiewicza 6, 30-348, Krak\'ow, Poland}
\email{marcin.sroka@uj.edu.pl}
\thanks{Research of the first author was supported in part by NSERC Discovery Grant, research of the second author was supported in part by  National Science Center of Poland grant no. 2021/41/B/ST1/01632.}

\dedicatory{Dedicated to the centennial of the birth of Louis Nirenberg (1925–2020)}

\begin{abstract} We establish a special concavity property for positive Hessian quotient operators $\frac{\sigma_n(W)}{\sigma_{n-k}(W)}, \ 1\le k\le n-1$. As a consequence, we prove a Jacobi inequality for general symmetric tensor satisfying positive Hessian quotient equation on Riemannian manifolds. \end{abstract}

\maketitle

\section{Introduction}
Concavity properties play vital role in regularity estimates for geometric fully nonlinear elliptic partial differential equations. An important class of these type equations consists of Hessian equations \[\sigma_k(W)=f\] for $1 \le k\le n$ and Hessian quotient equations \[Q_{k,l}:=\frac{\sigma_k(W)}{\sigma_{l}(W)}=f\]
for $ 1\le l<k\le n$. In the above $\sigma_k$ is the $k$-th elementary symmetric polynomial in $\mathbb R^n$ and for any $n\times n$ symmetric matrix $W$, $\sigma_k(W)$ is defined as the symmetric function of the eigenvalues of $W$. 

We recall G\r{a}rding's $k$-cone
\[\Gamma_k=\{\lambda=(\lambda_1,\cdots, \lambda_n)\in \mathbb R^n \ | \ \forall_{1\le m\le k} \ \sigma_m(\lambda)>0\}.\]
G\r{a}rding's hyperbolic polynomial theory yields the concavity of operators 
\begin{center} $\sigma^{\frac1k}_k(W)$ and $\left(\frac{\sigma_k(W)}{\sigma_l(W)}\right)^{\frac1{k-l}}$\end{center} for $W\in \Gamma_k$, see, e.g., \cite{CNS85}. 

For a smooth function $F(W)$ of $W=(w_{\alpha\beta})$, we define
\[F^{\alpha\beta}(W)=\frac{\partial F(W)}{\partial w_{\alpha\beta}}, \quad
F^{\alpha\beta,\gamma \eta}(W)=\frac{\partial^2 F(W)}{\partial w_{\alpha\beta}\partial{w_{\gamma \eta}}}.\]
Provided $F$ is concave in $W$ we have for any symmetric $\xi=\left(\xi_{\alpha \beta}\right)$:
\[F^{\alpha\beta,\gamma \eta}\xi_{\alpha \beta}\xi_{\gamma\eta}\le 0. \]
In the case $F=Q^{\frac1{k-l}}_{k,l}=\left(\frac{\sigma_{k}}{\sigma_{l}}\right)^{\frac1{k-l}}$, this is equivalent to
\begin{equation}\label{oc}Q_{k,l}^{\alpha\beta,\gamma\eta}(W)\xi_{\alpha \beta}\xi_{\gamma\eta}\le \left(1-\frac1{k-l}\right)\frac{\left(Q_{k,l}^{\alpha\beta}(W)\xi_{\alpha \beta}\right)^2}{Q_{k,l}(W)}.\end{equation}
In many cases, it is desirable to have a better term for controlling $\frac{w^2_{11,1}}{w_{11}}$ where $w_{11}$ is the largest
eigenvalue of $W$, for example the Jacobi inequalities in \cite{Lu24, S25}. In this article, we establish the following special concavity property of $F=\frac{\sigma_n}{\sigma_{n-k}}$ for $1\le k \le n-1$, an important class of nonlinear differential operators in geometric analysis.

\begin{theorem} \label{concavityproperty-k} For $F(W)=\frac{\sigma_n}{\sigma_{n-k}}(W), \ 1\le k\le n-1$, where $W\in \Gamma_n$ is assumed to be diagonal with eigenvalues $\lambda_1\ge \lambda_2\ge \cdots\ge  \lambda_n>0$, and for any $1\ge \tilde \delta>0$, there exists positive constant $ \epsilon_0$ depending only on $n, k, \tilde \delta$ such that for any symmetric $\xi_{\alpha \beta}$, $1\leq \alpha, \beta \leq n$:
\begin{eqnarray} \label{fineconcavity-k} 
- F^{\alpha\beta,\gamma \eta}\xi_{\alpha \beta}\xi_{\gamma\eta} &\ge &
 (1+\epsilon_0) F^{11} \frac{\xi_{11}^2}{\lambda_1} 
 +(1-\tilde \delta) \sum_{\alpha \ge 2} F^{\alpha \alpha}\frac{\xi^2_{\alpha \alpha}}{\lambda_\alpha}
 - \frac{\left( \sum_{\alpha=1}^n F^{\alpha \alpha}\xi_{\alpha \alpha}\right)^2}{F}\nonumber \\
 && +\sum\limits_{\substack{\alpha, \gamma \ge 2 \\ \alpha \neq \gamma}}F^{\alpha \alpha} \frac{\xi^2_{\alpha \gamma}}{\lambda_\gamma}+
 \left(1+\frac{n-k}{n-1}\right)\sum_{\alpha \ge 2}F^{\alpha \alpha}\frac{\xi^2_{\alpha 1}}{\lambda_1}.\end{eqnarray} 
 In particular, in the context of Section \ref{jacobi}, for every $1\le i \le n$ and $\xi_{\alpha \beta}=\nabla_iw_{\alpha \beta}=w_{\alpha\beta, i}$, we have
\begin{eqnarray} \label{prejacobi}
- F^{\alpha\beta,\gamma \eta}\xi_{\alpha \beta}\xi_{\gamma\eta} &\ge &
 (1+\epsilon_0) F^{11} \frac{w_{11,i}^2}{\lambda_1} 
 +(1-\tilde \delta) \sum_{\alpha \ge 2} F^{\alpha \alpha}\frac{w^2_{\alpha \alpha, i}}{\lambda_\alpha}
 - \frac{F_i^2}{F}\nonumber .
 \end{eqnarray}
\end{theorem}

\begin{remark} Assumption $k\le n-1$ is required in Theorem \ref{concavityproperty-k}. For $F=\sigma_n$, if $W$ is diagonal,
\[- F^{\alpha\beta,\gamma \eta}\xi_{\alpha \beta}\xi_{\gamma\eta}=
\sum_{\alpha, \beta=1}^n\frac{F^{\alpha\alpha}\xi_{\alpha\beta}^2}{w_{\beta\beta}}-\frac{\left(\sum_{\alpha=1}^nF^{\alpha \alpha}\xi_{\alpha \alpha}\right)^2}{F}.  \]
Simply letting \[\xi_{11}\neq 0 \ \mbox{and} \ \xi_{\alpha\beta}=0 \ \mbox{otherwise}\] indicates that the special concavity property (\ref{fineconcavity-k}) does not hold when $k=n$, i.e. for the Monge-Amp\`ere operator, in general. 
\end{remark}

The special concavity (\ref{fineconcavity-k}) is related to the following special property of $\sigma_k(W)$, for $W\in \Gamma_n$, which was proved in \cite{GLZ09} using Lemma 2.4 in \cite{GM03}.

\begin{lemma}\label{GLZ}[Guan-Li-Zhang] For $W\in \Gamma_n$, symmetric $\xi=\left(\xi_{\alpha\beta}\right)$ and $1\le k\le n$, it holds
\begin{equation} 
 \left(\sigma_{k}^{\alpha\beta,\gamma\eta}(W)+w^{\alpha\eta}\sigma_{k}^{\beta\gamma}(W)\right)\xi_{\alpha \beta}\xi_{\gamma\eta}\ge \frac{\left(\sum_{\alpha=1}^n\sigma_k^{\alpha \alpha}\xi_{\alpha \alpha}\right)^2}{\sigma_k(W)}\end{equation}
where $(w^{\alpha\beta})=W^{-1}$.\end{lemma}

Lemma 2.4 in \cite{GM03} and Lemma \ref{GLZ} are crucial for the Constant Rank Theorems in \cite{GLZ09}. It is not coincident that one may relate the upper bound of the eigenvalues of $W$ for \[\frac{\sigma_n(W)}{\sigma_{n-k}(W)}=f\] to the lower bound of the eigenvalues of $\widetilde{W}=W^{-1}$ for \[\sigma_{k}(\widetilde{W})=\frac1f.\] 

We refer to \cite{Lu24, S25} for some applications of Theorem \ref{concavityproperty-k}. 

\section{Proof of Theorem \ref{concavityproperty-k}
}

\begin{proof}
We set the notation. Let
\begin{eqnarray}  \begin{gathered} W=(w_{\alpha \beta}),\\
\widetilde{W}=W^{-1}\end{gathered}\end{eqnarray}
and, as always, the elements of the inverse matrix have the upper indexes. Moreover let 
\begin{eqnarray} \label{orderedlambda} \lambda_1 \geq \lambda_2 \geq ... \geq \lambda_n>0 \end{eqnarray} be the eigenvalues of $W$ and we set $\kappa_\alpha = \frac{1}{\lambda_\alpha}$ for $1 \le \alpha \le n$. Then, of course
\begin{eqnarray} \label{orderedkappa}  0 < \kappa_1 \leq \kappa_2 \leq ... \leq \kappa_n \end{eqnarray} and one observes that the operator can be rewritten as
\begin{eqnarray}   F(W)= \frac{\sigma_n(W)}{\sigma_{n-k}(W)}=\frac{1}{\sigma_k(\widetilde{W})}.\end{eqnarray} 

A straightforward computation, utilizing in particular the formula for the derivatives of the inverse matrix, yields
\begin{eqnarray} 
F^{\alpha\beta}(W) =
\delta^\alpha_\beta \frac{\sigma_{k-1}(\kappa|\alpha)}{\sigma_k^2(\widetilde{W})}\kappa_\alpha^2, 
\end{eqnarray}
\begin{eqnarray}
\sigma_k^2(\widetilde{W})F^{\alpha \beta,\gamma \eta}(W) &=& 
-\sigma_{k-2}(\kappa|\alpha \gamma)\delta_\alpha^\beta\delta_\gamma^\eta \kappa^2_{\alpha}\kappa^2_{\gamma}+\sigma_{k-2}(\kappa|\alpha \gamma)\delta_\alpha^\eta\delta_\beta^\gamma \kappa^2_{\alpha}\kappa^2_{\gamma}\nonumber
\\&& - \sigma_{k-1}(\kappa|\gamma)\delta_\beta^\gamma \delta^\alpha_\eta \kappa^2_{\gamma} \kappa_{\alpha}
- \sigma_{k-1}(\kappa|\alpha) \delta_\alpha^\eta \delta^\beta_\gamma \kappa^2_{\alpha}\kappa_{\gamma}\nonumber \\
&&+2 \sigma_{k-1}(\kappa|\alpha)\sigma_{k-1}(\kappa|\gamma)\delta_\alpha^\beta \delta_\gamma^\eta \frac{\kappa^2_{\alpha}\kappa^2_{\gamma}}{\sigma_k(\kappa)}.
\end{eqnarray}
From the above, after contracting with $\xi_{\alpha \beta}$, we obtain
\begin{eqnarray} \label{quadratic} \sigma_k^2(\kappa)F^{\alpha \beta,\gamma \eta}(W)\xi_{\alpha\beta}\xi_{\gamma\eta}& =&
-\sigma_{k-2}(\kappa|\alpha \gamma) \kappa^2_{\alpha}\kappa^2_{\gamma}\xi_{\alpha\alpha}\xi_{\gamma\gamma}
+\sigma_{k-2}(\kappa|\alpha \gamma)\kappa^2_{\alpha}\kappa^2_{\gamma} \xi_{\alpha\gamma}^2\nonumber 
\\&& - \sigma_{k-1}(\kappa|\gamma) \kappa^2_{\gamma}\kappa^2_{\alpha}\xi_{\alpha\gamma}^2
- \sigma_{k-1}(\kappa|\alpha) \kappa^2_{\alpha}\kappa^2_{\gamma}\xi_{\alpha\gamma}^2 
\nonumber \\
&&+2 \frac{\left(\sigma_{k-1}(\kappa|\alpha)\kappa^2_{\alpha}\xi_{\alpha\alpha}\right)^2}{\sigma_k(\kappa)}\nonumber \\
&=&
\Big[ -\sum\limits_{\alpha\not=\gamma} \sigma_{k-2}(\kappa|\alpha \gamma) \kappa^2_{\alpha}\kappa^2_{\gamma}{\gamma}\xi_{\alpha\alpha}\xi_{\gamma\gamma} 
-\sigma_{k-1}(\kappa|\alpha)  \kappa^3_{\alpha}\xi_{\alpha\alpha}^2 
\nonumber \\
&&+ \frac{\left(\sigma_{k-1}(\kappa|\alpha)\kappa^2_{\alpha}\xi_{\alpha\alpha}\right)^2}{\sigma_k(\kappa)}\Big]\nonumber \\
&& + \Big\{-\sigma_{k-1}(\kappa|\alpha)  \kappa_{\alpha}^3\xi_{\alpha\alpha}^2 
+ \frac{\left(\sigma_{k-1}(\kappa|\alpha)\kappa^2_{\alpha}\xi_{\alpha\alpha}\right)^2}{\sigma_k(\kappa)}\Big\}\nonumber \\
&&+ \Big( \sum\limits_{\alpha\not=\gamma} \sigma_{k-2}(\kappa|\alpha \gamma) \kappa_{\alpha}^2\kappa_{\gamma}^2\xi_{\alpha\gamma}^2 - \sum\limits_{\alpha\not=\gamma}\sigma_{k-1}(\kappa|\alpha)\kappa_{\alpha}^2\kappa_{\gamma}\xi_{\alpha\gamma}^2\Big).
  \end{eqnarray}

The point here is that the term in the square brackets has a sign. Moreover, it has a concrete and compact formula that follows from the calculations of \cite{GM03, GLZ09}. Namely, from formula $(5.9)$ and the two following ones in \cite{GLZ09} we obtain the expression for this term which, moreover, from the last formula in the proof of Lemma 2.4 in \cite{GM03} is seen to provide:
\begin{eqnarray} \label{guanconcav} \begin{gathered}
\Big[ -\sum\limits_{\alpha\not=\gamma} \sigma_{k-2}(\kappa|\alpha \gamma) \kappa_{\alpha}^2\kappa_{\gamma}^2\xi_{\alpha\alpha}\xi_{\gamma\gamma} 
-\sigma_{k-1}(\kappa|\alpha)  \kappa_{\alpha}^3\xi_{\alpha\alpha}^2 
+ \frac{\left(\sigma_{k-1}(\kappa|\alpha)\kappa_{\alpha}^2\xi_{\alpha\alpha}\right)^2}{\sigma_k(\kappa)}\Big]\\=
\frac{1}{2\sigma_k(\kappa)}\Big[ \sum\limits_{\alpha\not=\beta} \left(\sigma_{k-1}^2(\kappa|\alpha\beta)-\sigma_k(\kappa|\alpha\beta)\sigma_{k-2}(\kappa|\alpha\beta)\right)\kappa_{\alpha}\kappa_{\beta}\left(\kappa_{\alpha}\xi_{\alpha\alpha}-\kappa_{\beta}\xi_{\beta\beta}{}\right)^2\Big].
\end{gathered}\end{eqnarray}

Applying (\ref{guanconcav}) in (\ref{quadratic}) we obtain

\begin{eqnarray} \label{quadratic2} \begin{gathered}\sigma_k^2(\kappa)F^{\alpha \beta,\gamma \eta}(W)\xi_{\alpha\beta}\xi_{\gamma\eta} =\\
-\frac{1}{2\sigma_k(\kappa)}\Big[ \sum\limits_{\alpha\not=\beta} \left(\sigma_{k-1}^2(\kappa|\alpha\beta)-\sigma_k(\kappa|\alpha\beta)\sigma_{k-2}(\kappa|\alpha\beta)\right)\kappa_{\alpha}\kappa_{\beta}\left(\kappa_{\alpha}\xi_{\alpha\alpha}-\kappa_{\beta}\xi_{\beta\beta}\right)^2\Big]\\
+\Big\{-\sigma_{k-1}(\kappa|\alpha)  \kappa_{\alpha}^3\xi_{\alpha\alpha}^2 
+ \frac{\left(\sigma_{k-1}(\kappa|\alpha)\kappa_{\alpha}^2\xi_{\alpha\alpha}\right)^2}{\sigma_k(\kappa)}\Big\}\\
+ \Big( \sum\limits_{\alpha\not=\gamma} \sigma_{k-2}(\kappa|\alpha \gamma) \kappa_{\alpha}^2\kappa_{\gamma}^2 \xi_{\alpha\gamma}^2 - 2\sum\limits_{\alpha\not=\gamma}\sigma_{k-1}(\kappa|\alpha)\kappa_{\alpha}^2\kappa_{\gamma}^2\lambda_{\gamma}\xi_{\alpha\gamma}^2\Big).
\end{gathered}\end{eqnarray}

Let us denote:
\begin{eqnarray} \label{Is} 
\quad \quad I_1'&=&\frac{1}{2\sigma_k(\kappa)} \sum\limits_{\alpha\not=\beta} \left(\sigma_{k-1}^2(\kappa|\alpha\beta)-\sigma_k(\kappa|\alpha\beta)\sigma_{k-2}(\kappa|\alpha\beta)\right)\kappa_{\alpha}\kappa_{\beta}\left(\kappa_{\alpha}\xi_{\alpha\alpha}-\kappa_{\beta}\xi_{\beta\beta}\right)^2,\nonumber \\
I_2&=&\sigma_{k-1}(\kappa|\alpha)  \kappa_{\alpha}^3\xi_{\alpha\alpha}^2 
- \frac{\left(\sigma_{k-1}(\kappa|\alpha)\kappa_{\alpha}^2\xi_{\alpha\alpha}\right)^2}{\sigma_k(\kappa)},\\
I_3&=&-\sum\limits_{\alpha\not=\gamma} \sigma_{k-2}(\kappa|\alpha \gamma) \kappa_{\alpha}^2\kappa_{\gamma}^2 \xi_{\alpha\gamma}^2 + 2\sum\limits_{\alpha\not=\gamma}\sigma_{k-1}(\kappa|\alpha)\kappa_{\alpha}^2\kappa_{\gamma}^2\lambda_{\gamma}\xi_{\alpha\gamma}^2.\nonumber
\end{eqnarray}
 
By Newton-Maclaurin inequality for any $1\le \alpha \not = \beta \le n$:
\begin{eqnarray} \label{newtonmac} \sigma_{k-1}^2(\kappa|\alpha\beta)-\sigma_k(\kappa|\alpha\beta)\sigma_{k-2}(\kappa|\alpha\beta)\geq c_{n,k}\sigma_{k-1}(\kappa|\alpha\beta),\end{eqnarray}
while for any $1 \le \gamma \not= \alpha \le n$:
\begin{eqnarray}\label{sigmadecomposition} \lambda_{\gamma}\sigma_{k-1}(\kappa|\alpha)=\sigma_{k-2}(\kappa|\alpha\gamma)+\lambda_{\gamma}\sigma_{k-1}(\kappa|\alpha\gamma).\end{eqnarray}

Applying (\ref{newtonmac}) in $I_1'$ provides:
\begin{eqnarray} \begin{gathered}
I_1' \geq \frac{c_{n,k}}{2\sigma_k(\kappa)} \sum\limits_{\alpha\not=\beta} \sigma_{k-1}^2(\kappa|\alpha\beta) \kappa_{\alpha}\kappa_{\beta}\left(\kappa_{\alpha}\xi_{\alpha\alpha}-\kappa_{\beta}\xi_{\beta\beta}\right)^2:=I_1.
\end{gathered} \end{eqnarray}

While applying (\ref{sigmadecomposition}) for $I_3$ we get:
\begin{eqnarray} \label{I3} 
I_3=\sum\limits_{\alpha\not=\gamma} \Big(\sigma_{k-1}(\kappa|\alpha)+\sigma_{k-1}(\kappa|\alpha\gamma)\Big)\lambda_{\gamma}\kappa_{\alpha}^2\kappa_{\gamma}^2 \xi_{\alpha\gamma}^2.\end{eqnarray}

Thus we have obtained so far:
\begin{eqnarray} \label{consise} \sigma_k^2(\kappa)F^{\alpha \beta,\gamma \eta}(W)\xi_{\alpha\beta}\xi_{\gamma\eta} \leq -I_1-I_2-I_3.\end{eqnarray}

Let us continue the argument in following two cases for any $1 > \delta_0 >0$.

\textbf{Case 1.} Suppose there exists $\alpha > 1$ such that
\begin{eqnarray} \label{case1} \left| \frac{\xi_{\alpha\alpha}}{\lambda_\alpha}\right| \geq \delta_0 \left| \frac{\xi_{11}}{\lambda_1}\right|.\end{eqnarray} 
Then for any $1 \ge \tilde \delta>0$:   
\begin{eqnarray} \label{I2est}  \begin{gathered}
\frac{1}{\sigma_k^2(\kappa)}I_2=
\frac{\sigma_{k-1}(\kappa|\alpha)}{\sigma_k^2(\kappa)}  \kappa_{\alpha}^3\xi_{\alpha\alpha}^2 
- \left(\frac{\sigma_{k-1}(\kappa|\alpha)}{\sigma_k^2(\kappa)}\kappa_{\alpha}^2\xi_{\alpha\alpha}\right)^2\sigma_k(\kappa)\\
=F^{\alpha\alpha}(\lambda)  \frac{\xi_{\alpha\alpha}^2}{\lambda_\alpha} 
- \frac{1}{F(\lambda)} \Big( F^{\alpha\alpha}(\lambda) \xi_{\alpha\alpha}\Big)^2\\
=F^{11}(\lambda)  \frac{\xi_{11}^2}{\lambda_1}+\tilde \delta \sum\limits_{\alpha \geq 2}F^{\alpha\alpha}(\lambda)  \frac{\xi_{\alpha\alpha}^2}{\lambda_\alpha} \\
+(1-\tilde \delta) \sum\limits_{\alpha \geq 2} F^{\alpha\alpha}(\lambda)  \frac{\xi_{\alpha\alpha}^2}{\lambda_\alpha} 
- \frac{1}{F(\lambda)} \Big( F^{\alpha\alpha}(\lambda) \xi_{\alpha\alpha}\Big)^2\\
\geq \Big(1+\tilde \delta\delta_0^2 \Big)F^{11}(\lambda)  \frac{\xi_{11}^2}{\lambda_1}
+(1-\tilde \delta) \sum\limits_{\alpha \geq 2}F^{\alpha\alpha}(\lambda)  \frac{\xi_{\alpha\alpha}^2}{\lambda_\alpha}
- \frac{1}{F(\lambda)} \Big( F^{\alpha\alpha}(\lambda) \xi_{\alpha\alpha}\Big)^2
\end{gathered}\end{eqnarray}
where the last inequality follows from, for every $\alpha$ satisfying (\ref{case1}),
\begin{eqnarray} \begin{gathered} 
F^{\alpha \alpha}(\lambda) \frac{\xi_{\alpha\alpha}^2}{\lambda_\alpha} = F^{\alpha \alpha}(\lambda) \frac{\xi_{\alpha\alpha}^2}{\lambda_{\alpha}^2} \cdot \lambda_\alpha \geq \delta_0^2 \lambda_1 F^{11}(\lambda) \frac{\xi_{11}^2}{\lambda_1^2}=\delta_0^2 F^{11}(\lambda) \frac{\xi_{11}^2}{\lambda_1}.
\end{gathered}\end{eqnarray}
In the above calculation, we utilized that for any $\alpha >1$: 
\begin{eqnarray} \label{fif1comp} \begin{gathered} 
F^{\alpha \alpha}(\lambda) \frac{\lambda_\alpha}{\lambda_1} = \frac{\sigma_{k-1}(\kappa|\alpha)}{\sigma_k^2(\kappa)} \cdot \frac{1}{\lambda_\alpha\lambda_1} =  \frac{\sigma_{k-1}(\kappa|\alpha)\kappa_\alpha \kappa_1}{\sigma_k^2(\kappa)} 
\\ = \frac{\kappa_\alpha \kappa_1(\sigma_{k-1}(\kappa|\alpha 1) + \kappa_1 \sigma_{k-2}(\kappa|\alpha 1))}{\sigma_k^2(\kappa)} = \frac{ \kappa_1(\sigma_{k-1}(\kappa|\alpha 1)\kappa_\alpha + \kappa_\alpha \kappa_1 \sigma_{k-2}(\kappa|\alpha 1))}{\sigma_k^2(\kappa)}
\\ \geq \frac{ \kappa_1(\sigma_{k-1}(\kappa|\alpha 1)\kappa_1 + \kappa_\alpha \kappa_1 \sigma_{k-2}(\kappa|\alpha 1))}{\sigma_k^2(\kappa)} = \frac{ \kappa_1^2(\sigma_{k-1}(\kappa|\alpha 1) + \kappa_\alpha \sigma_{k-2}(\kappa|\alpha 1))}{\sigma_k^2(\kappa)} 
\\= \frac{ \kappa_1^2 \sigma_{k-1}(\kappa| 1)}{\sigma_k^2(\kappa)} = F^{11}(\lambda).
\end{gathered}\end{eqnarray}

Applying (\ref{I2est}) in (\ref{consise}) we obtain
\begin{eqnarray} \label{consisecase1} \begin{gathered}
-F^{\alpha \beta,\gamma \eta}(\lambda)\xi_{\alpha\beta}\xi_{\gamma\eta} \geq 
\Big(1+\tilde \delta\delta_0^2 \Big)F^{11}(\lambda)  \frac{\xi_{11}^2}{\lambda_1}
+(1-\tilde \delta) \sum\limits_{\alpha \geq 2}F^{\alpha\alpha}(\lambda)  \frac{\xi_{\alpha\alpha}^2}{\lambda_\alpha}\\
- \frac{1}{F(\lambda)} \Big( F^{\alpha\alpha}(\lambda) \xi_{\alpha\alpha}\Big)^2 +F^2(\lambda)\Big(I_1+I_3\Big).
\end{gathered}\end{eqnarray}

\textbf{Case 2.} As the complementary case we assume now that for every $\alpha > 1$: 
\begin{eqnarray} \label{a2} \left| \frac{\xi_{\alpha\alpha}}{\lambda_\alpha}\right| < \delta_0 \left| \frac{\xi_{11}}{\lambda_1}\right|.\end{eqnarray}

In this case, we start with the estimate

\begin{eqnarray} \label{estI1} \begin{gathered}
I_1 \geq \frac{c_{n,k}}{2\sigma_k(\kappa)} \sum\limits_{\alpha > 1} \sigma_{k-1}^2(\kappa|\alpha 1) \kappa_{\alpha}\kappa_{1}\left(\kappa_{\alpha}\xi_{\alpha\alpha}-\kappa_{\beta}\xi_{\beta\beta}\right)^2\\
\geq \frac{c_{n,k}}{2\sigma_k(\kappa)} \sum\limits_{\alpha > 1} \sigma_{k-1}^2(\kappa|\alpha 1) \kappa_{\alpha}(1-\delta_0)^2 \frac{\xi_{11}^2}{\lambda_1^3}.
\end{gathered}\end{eqnarray}

For $1\le k\le n-1$ and $\alpha=n-k+1\ge 2$, from the properties of $\sigma_k$ operator and (\ref{orderedkappa}), we get:
\begin{eqnarray}  
\sigma_{k-1}(\kappa|\alpha 1) \kappa_{\alpha} \geq \widetilde{c_{n,k}} \sigma_{k}(\kappa), \quad 
\sigma_{k-1}(\kappa|\alpha 1) \geq \widetilde{\widetilde{c_{n,k}}}\sigma_{k-1}(\kappa|1).
\end{eqnarray}

Applying the above inequalities in (\ref{estI1}) we obtain:
\begin{eqnarray} \label{estI1second} \begin{gathered}
I_1 \geq c_{n,k}'(1-\delta_0)^2 \frac{\sigma_{k-1}(\kappa|1)}{\sigma_k^2(\kappa)}\kappa_1^2 \frac{\xi_{11}^2}{\lambda_1}\sigma_k^2(\kappa)= c_{n,k}'(1-\delta_0)^2F^{11}(\lambda)\frac{\xi_{11}^2}{\lambda_1}\sigma_k^2(\kappa)
\end{gathered}.\end{eqnarray}
Note that $c^{'}_{n,k}$ can be tracked explicitly. 

Applying (\ref{estI1second}) in (\ref{consise}) and estimating $I_2$ similarly as in Case 1 we get:
\begin{eqnarray} \label{consise2} \begin{gathered}
-F^{\alpha \beta,\gamma \eta}(\lambda)\xi_{\alpha\beta}\xi_{\gamma\eta} \geq 
\Big(1+c_{n,k}'(1-\delta_0)^2\Big)F^{11}(\lambda)\frac{\xi_{11}^2}{\lambda_1}\\ 
+\sum\limits_{\alpha \geq 2}F^{\alpha\alpha}(\lambda)  \frac{\xi_{\alpha\alpha}^2}{\lambda_\alpha}
- \frac{1}{F(\lambda)} \Big( F^{\alpha\alpha}(\lambda) \xi_{\alpha\alpha}\Big)^2 +I_3F^2(\lambda).\end{gathered}\end{eqnarray}

\bigskip

Pick $\delta_0$ such that 
\[c_{n,k}'(1-\delta_0)^2=\tilde \delta \delta_0^2:=\epsilon^2_{n,k,\tilde \delta}.\] In both Cases 1 and 2 we arrive at
\begin{eqnarray} \label{consisefinal} \begin{gathered}
-F^{\alpha \beta,\gamma \eta}(\lambda)\xi_{\alpha\beta}\xi_{\gamma\eta} \geq 
\Big(1+\epsilon_{n,k,\tilde \delta}^2\Big)F^{11}(\lambda)\frac{\xi_{11}^2}{\lambda_1}\\ 
+(1-\tilde \delta)\sum\limits_{\alpha \geq 2}F^{\alpha\alpha}(\lambda)  \frac{\xi_{\alpha\alpha}^2}{\lambda_\alpha}
- \frac{1}{F(\lambda)} \Big( F^{\alpha\alpha}(\lambda) \xi_{\alpha\alpha}\Big)^2 +I_3F^2(\lambda).\end{gathered}\end{eqnarray}

We see thus in order to finish the proof of the first part of the claim it is enough to bound $I_3$ further. Note that the case $n=2$ is trivial here, we may assume $n\ge 3$. For that goal we compute:

\begin{eqnarray} \label{I3decomp} \begin{gathered}
F^2(\lambda)I_3 = 
\sum\limits_{\alpha\not=\gamma} \frac{\sigma_{k-1}(\kappa|\alpha)}{\sigma_k^2(\kappa)}\kappa_\alpha^2 \frac{\xi_{\alpha\gamma}^2}{\lambda_\gamma} 
+\sum\limits_{\alpha\not=\gamma}  \frac{\sigma_{k-1}(\kappa|\alpha\gamma)}{\sigma_k^2(\kappa)} \kappa_\alpha^2 \frac{\xi_{\alpha\gamma}^2}{\lambda_\gamma} \\ 
\geq \sum\limits_{\alpha\not=\gamma} F^{\alpha\alpha}(\lambda) \frac{\xi_{\alpha\gamma}^2}{\lambda_\gamma}+
\sum\limits_{\alpha >1 }  \frac{\sigma_{k-1}(\kappa|\alpha 1)}{\sigma_k^2(\kappa)} \kappa_\alpha^2 \frac{\xi_{\alpha 1}^2}{\lambda_1} \\ = J_1+J_2+J_3+K_1 \end{gathered} \end{eqnarray}
where
\begin{eqnarray} \label{Js} \begin{gathered}
J_1= \sum\limits_{\alpha \geq 2} F^{\alpha\alpha}(\lambda) \frac{\xi_{\alpha 1}^2}{\lambda_1},\\
J_2= \sum\limits_{\gamma \geq 2} F^{11}(\lambda) \frac{\xi_{1 \gamma}^2}{\lambda_\gamma},\\
J_3=\sum\limits_{\substack{\alpha\not=\gamma \\ \alpha, \gamma \geq 2}} F^{\alpha\alpha}(\lambda) \frac{\xi_{\alpha\gamma}^2}{\lambda_\gamma},\\
K_1=\sum\limits_{\alpha >1 }  \frac{\sigma_{k-1}(\kappa|\alpha 1)}{\sigma_k^2(\kappa)} \kappa_\alpha^2 \frac{\xi_{\alpha 1}^2}{\lambda_1}.\end{gathered} \end{eqnarray}
We note that thanks to (\ref{orderedkappa}):
\begin{eqnarray}  \sigma_{k-1}(\kappa|\alpha 1)\ge \frac{n-k}{n-1} \sigma_{k-1}(\kappa|\alpha) \text{ for any } \alpha>1,\end{eqnarray}
resulting in
\begin{eqnarray} \label{j1k1} \begin{gathered}
     J_1+K_1 \geq \sum\limits_{\alpha \geq 2} \Big(F^{\alpha\alpha}(\lambda)+F^{\alpha \alpha}(\lambda)\frac{n-k}{n-1}  \Big)\frac{\xi_{\alpha 1}^2}{\lambda_1}
\\ \geq 
\left(1+\frac{n-k}{n-1} \right)\sum\limits_{\alpha \geq 2} F^{\alpha\alpha}(\lambda) \frac{\xi_{\alpha 1}^2}{\lambda_1}.\end{gathered}\end{eqnarray}
We also note that due to (\ref{fif1comp}):
\begin{eqnarray} \label{j1j2} \begin{gathered}
J_1= \sum\limits_{\alpha \geq 2} F^{\alpha\alpha}(\lambda) \frac{\xi_{\alpha 1}^2}{\lambda_1}\geq 
\sum\limits_{\alpha \geq 2} F^{11}(\lambda) \frac{\xi_{1 \alpha}^2}{\lambda_\alpha}=J_2\end{gathered} \end{eqnarray}
thus (\ref{j1k1}) is qualitatively sharp regarding the terms $\xi_{1\alpha}$ for $\alpha>1$.

Finally, (\ref{j1k1}) together with (\ref{consisefinal}) and (\ref{I3decomp}) gives (\ref{fineconcavity-k}). 
\end{proof}

\section{Jacobi inequality}\label{jacobi}

For this section, let $(M,g,W)$ be an $n$ dimensional Riemannian manifold $(M,g)$ endowed with a smooth, symmetric $(0,2)$ tensor $W$. For the whole section, we denote by $\nabla$ the Levi-Civita connection of $g$. All derivatives below are covariant derivatives with respect to $\nabla$. Using $g$ and $W$ we introduce the endomorphism field
\begin{eqnarray}
    W_g:=g^{-1} \circ W : TM \xrightarrow{W} T^*M \xrightarrow{g^{-1}} TM. 
\end{eqnarray}
Later, by abuse of notation, we still denote $W_g$ by $W$. Let 
\[ \lambda(W_g)=(\lambda_1,...,\lambda_n)\] be a decreasingly ordered vector of eigenvalues of $W_g$. Again, in the computations we will suppress the argument and simply write $\lambda$ and assume $\lambda \in \Gamma_n$.
We still denote \begin{eqnarray*}
    F(W_g)=\frac{\sigma_n(W_g)}{\sigma_{n-k}(W_g)}
\end{eqnarray*}
for $1\le k \le n-1$. We consider the equation
\begin{eqnarray}\label{WE}
    F\left(W_g(x)\right)=f(x)
\end{eqnarray}
for $ x\in M$.

In the rest of this section, all computations will be done with respect to local orthonormal frames.

\begin{definition}
    We say the tensor $W$ on $(M,g)$ satisfies Condition $(A)$,  provided for an uniform constant $A>0$ independent of $W$, the inequalities
    \begin{eqnarray}\label{CA} \quad
    \begin{gathered}
        |\nabla_k w_{ij}(x)-\nabla_i w_{jk}(x)| \le A\left(1+|W(x)|_g\right),   \\
        |\nabla^2_{kl} w_{ij}(x)-\nabla^2_{ij}w_{kl}(x)| \le 
        A ( 1 + |W(x)|_g \\
        + |\nabla w_{ik}(x)|_g+|\nabla w_{jk}(x)|_g+|\nabla w_{il}(x)|_g+|\nabla w_{jl}(x)|_g)\end{gathered}
    \end{eqnarray}
    hold in any orthonormal frame and  at every point $x$ in $M$.
\end{definition}

\begin{definition}
We call a symmetric tensor $W$ a Q-Codazzi tensor, for a positive constant $Q$ independent of $W$, if there exist tensors $c$ and $q$ such that the equality
    \begin{eqnarray}\label{QC}\nabla_i w_{kl}-\nabla_k w_{il}=\sum_{m=1}^nc_{ik}^mw_{ m l}+q_{ikl}\end{eqnarray}
holds, where  \begin{eqnarray}\label{CC}|c_{ik}^m(x)|+|\nabla c_{ik}^m(x)|_g \le Q, \quad |q_{ikl}(x)|+ |\nabla q_{ikl}(x)|_g \le Q (1+|W(x)|_g)\end{eqnarray} for any orthonormal frame and at every point $x$ in $M$. \end{definition}

By Ricci identity, any Q-Codazzi tensor $W$ satisfies Condition (A). 
\begin{remark}The following is a list of examples of Q-Codazzi tensors.\begin{enumerate} \item  Codazzi tensors. \item Second fundamental forms of immersed hypersurfaces in Riemannian ambient space. 
\item $\nabla^2 u+\chi$ for any given symmetric $(0,2)$ tensor $\chi$ provided that $\nabla u$ is under control. 

\end{enumerate}\end{remark}

\begin{proposition}\label{Pb}
    Let $(M,g,W)$ and $F$ be as above. Assume that $W$ satisfies equation (\ref{WE}) and Condition (A). There exist $\epsilon(n,k)>0$ (depending only on $n, k$) and a constant $C$ (depending on $n, k, f$ and $A$ in (\ref{CA})) such that the following differential inequality holds for $b=\log (1+\lambda_1)$  in viscosity sense
    \begin{eqnarray}\label{Jb}
        \sum_{\alpha,\beta=1}^n F^{\alpha \beta}\left(\lambda(x)\right)b_{\alpha \beta}(x) &\ge& \epsilon(n,k) \sum_{\alpha,\beta=1}^n F^{\alpha \beta}\left(\lambda(x)\right) b_\alpha(x) b_\beta(x) \nonumber \\
        & &-C(n,k,f,A)  \mathcal F(x)-\frac{|\nabla f(x)|_g^2}{f(x)(\lambda_1(x)+1)}-\frac{|\nabla^2 f(x)|_g}{\lambda_1(x)+1}
    \end{eqnarray}
    at any $x\in M$, where $\mathcal F(x)=\sum\limits_{\alpha=1}^n F^{\alpha \alpha}(\lambda(x))$.
\end{proposition}

\begin{proof}
    We choose the normal coordinates around $x \in M$ such that $W_g(x)$ is diagonal. The standard calculations yield $ \lambda_{1,\alpha}=w_{11,\alpha}$ and
    \begin{eqnarray}
        \label{lambdasecondest} \lambda_{1.\alpha \alpha} \geq w_{11,\alpha \alpha}+2\sum_{\beta>\mu}\frac{w^2_{1\beta,\alpha}}{\lambda_1-\lambda_\beta}
    \end{eqnarray}
     in viscosity sense for $1\le \mu<n$ being the multiplicity of $\lambda_1$. We work only on the case $\mu=1$, and the argument below can be easily adapted for the general case $\mu\ge 1$. 
    
    In the rest of this section, a constant $C$ may change line by line while staying under control. We also note that $|W|_g \le n \lambda_1$ and
    \begin{eqnarray}
        \lambda_1 \ge C_{n,k} \inf_{x \in M} f^\frac{1}{k}(x) \ge C.
    \end{eqnarray}

Applying (\ref{lambdasecondest}), Condition $(A)$ for $W$ and Cauchy-Schwarz inequality gives
\begin{eqnarray}\label{fourthest}\begin{gathered}
    \sum_{\alpha=1}^nF^{\alpha\alpha}\frac{\lambda_{1,\alpha \alpha}}{\lambda_1+1} 
    \ge \sum_{\alpha=1}^nF^{\alpha\alpha}\frac{w_{11,\alpha\alpha}}{\lambda_1+1}+2\sum_{\alpha>1}F^{\alpha\alpha}\frac{w^2_{1\alpha,\alpha}}{\lambda_1(\lambda_1+1)}\\
    \ge \sum_{\alpha=1}^nF^{\alpha\alpha}\frac{w_{\alpha\alpha,11}-C\sum_m|w_{1\alpha,m}|-C(\lambda_1+1)}{\lambda_1+1}+\sum_{\alpha>1}F^{\alpha\alpha}\frac{w^2_{\alpha\alpha,1}}{\lambda_1(\lambda_1+1)}-C \mathcal F\\
    \ge -F^{\alpha \beta, \gamma \delta}\frac{w_{\alpha \beta,1}w_{\gamma \delta, 1}}{\lambda_1+1}+\frac{F_{11}}{\lambda_1+1}-C \mathcal F
   - C\sum_{\alpha=1}^nF^{\alpha \alpha}\frac{|w_{11,\alpha}|+|w_{\alpha\alpha,1}|}{\lambda_1+1}\\
   -C\sum\limits_{\substack{\alpha, \gamma \ge 2 \\ \alpha \neq \gamma}}
  F^{\alpha \alpha}\frac{|w_{\alpha\gamma,1}|}{\lambda_1+1}+\sum_{\alpha>1}F^{\alpha\alpha}\frac{w^2_{\alpha\alpha,1}}{\lambda_1(\lambda_1+1)}\\
  \ge -F^{\alpha \beta, \gamma \delta}\frac{w_{\alpha \beta,1}w_{\gamma \delta, 1}}{\lambda_1+1}+\frac{F_{11}}{\lambda_1+1}-C \mathcal F
   \\- C\sum_{\alpha=1}^nF^{\alpha \alpha}\frac{|w_{11,\alpha}|}{\lambda_1+1}
   -\sum\limits_{\substack{\alpha, \gamma \ge 2 \\ \alpha \neq \gamma}}F^{\alpha \alpha} \frac{w^2_{\alpha \gamma,1}}{\lambda_\gamma(\lambda_1+1)}.
\end{gathered}\end{eqnarray}

Applying (\ref{fineconcavity-k}) for $\xi_{\alpha\beta}=w_{\alpha\beta,1}$ we obtain 
\begin{eqnarray}\label{concavappl}
    - F^{\alpha\beta,\gamma \eta}w_{\alpha \beta,1}w_{\gamma\eta,1} &\ge &
 (1+\epsilon_0) F^{11} \frac{w_{11,1}^2}{\lambda_1}- \frac{\left(\sum_{\alpha=1}^nF^{\alpha \alpha}w_{\alpha \alpha,1}\right)^2}{F} \\
 && +\left(1+\frac{n-k}{n-1}\right)\sum_{\alpha \ge 2}F^{\alpha \alpha}\frac{w^2_{\alpha 1,1}}{\lambda_1}
 +\sum\limits_{\substack{\alpha, \gamma \ge 2 \\ \alpha \neq \gamma}}F^{\alpha \alpha} \frac{w^2_{\alpha \gamma,1}}{\lambda_\gamma}.\nonumber
\end{eqnarray}

Using again Condition (A), we see that for any $1 \le \alpha \le n$ and $1 >>\theta>0$:
\begin{eqnarray}\label{thirdordperm}
    w_{\alpha 1,1}^2\ge (1-\theta)w_{11,\alpha}^2-C_\theta(\lambda_1+1)^2.
\end{eqnarray}

Set $\epsilon_0'= \min \left\{ \epsilon_0,\frac{n-k}{n-1}\right\}$, by Cauchy-Schwarz inequality,
\begin{eqnarray}\label{thirdlinearest}
    -C\sum_{\alpha=1}^nF^{\alpha \alpha} \frac{|w_{11,\alpha}|}{\lambda_1+1} \ge -\sum_{\alpha=1}^nF^{\alpha \alpha}\left(\frac14 \epsilon_0' \left(\frac{w_{11,\alpha}}{\lambda_1+1}\right)^2+C\right).
\end{eqnarray}

Applying (\ref{concavappl}) - (\ref{thirdlinearest}) in (\ref{fourthest}) gives for $\epsilon_0'= \min \left\{ \epsilon_0,\frac{n-k}{n-1}\right\}$, for $\epsilon_0$ from Theorem \ref{concavityproperty-k},
\begin{eqnarray}\begin{gathered}
    F^{\alpha\alpha}\frac{\lambda_{1,\alpha \alpha}}{\lambda_1+1} \ge 
    \left(1+\frac12\epsilon_0'\right) F^{11} \frac{w_{11,1}^2}{\lambda_1(\lambda_1+1)}
 +\left(1+\frac12\epsilon_0'\right)\sum_{\alpha \ge 2}F^{\alpha \alpha}\frac{w^2_{11,\alpha}}{\lambda_1(\lambda_1+1)}\\
 - \frac{F_1^2}{F(\lambda_1+1)} +\frac{F_{11}}{\lambda_1+1}-C \mathcal F.\end{gathered}
\end{eqnarray}
This translates into
\begin{eqnarray}
    F^{\alpha\alpha}b_{\alpha \alpha} \ge 
    \frac12\epsilon_0' F^{\alpha \alpha}b_\alpha^2
 - \frac{F_1^2}{F(\lambda_1+1)} +\frac{F_{11}}{\lambda_1+1}-C \mathcal F
\end{eqnarray}
as required.
\end{proof}
\bigskip

\begin{remark} Under the same assumptions as in Proposition \ref{Pb}, let $$B(x)=\log \left(D+\log (\lambda_1(x)+1) \right)$$ for some $D>0$ depending on $n,k,f,A$. Then, the Jacobi inequality (\ref{Jb}) implies
\begin{eqnarray}\label{JB}
        \sum_{\alpha,\beta=1}^n F^{\alpha \beta}\left(\lambda(x)\right)B_{\alpha \beta}(x) &\ge& \frac{\epsilon(n,k)}2\left(D+\log (\lambda_1(x)+1)\right) \sum_{\alpha,\beta=1}^n F^{\alpha \beta}\left(\lambda(x)\right) B_\alpha(x) B_\beta(x) \nonumber \\
        & &-\frac{1}{\sqrt{D+\log (\lambda_1(x)+1)}} \left(\mathcal F(x)+\frac{|\nabla f(x)|_g^2}{f(x)(\lambda_1(x)+1)}+\frac{|\nabla^2 f(x)|_g}{\lambda_1(x)+1}\right).
    \end{eqnarray}
    \begin{proof}We have \[B_\alpha=\frac{b_\alpha}{D+b}, \quad B_{\alpha \beta}=\frac{b_{\alpha\beta}}{D+b}-\frac{b_\alpha b_\beta}{(D+b)^2}.\] If we pick $D$ sufficiently large such that 
    \[\sqrt{D+\log (\lambda_1+1)}>\frac{2}{\epsilon(n,k)}+C(n,k,f,A),\] then (\ref{JB}) follows from (\ref{Jb}). \end{proof}
\end{remark}
\begin{remark} Suppose the first inequality in (\ref{CA}) does not involve the zero order term. In this case the constants $C$ and $D$ in (\ref{Jb}) and (\ref{JB}) are independent of $f$. Moreover, if $f>0$ and $f\in C^2(M)$, then there is a constant $C$ depending only on $\|f\|_{C^2(M)}$ and $(M,g)$ such that \[|\nabla f(x)|^2\le C f(x)\] for any $x\in M$. As such, in this setting, (\ref{Jb}) and (\ref{JB}) are independent of the lower bound of $f$.   \end{remark}


\begin{thebibliography}{99}
\bibitem{CNS85} L. Caffarelli, L. Nirenberg, and J. Spruck, \textit{The Dirichlet problem for nonlinear second--order elliptic equations. III. Functions of the eigenvalues of the Hessian.}, Acta Math., \textbf{155}(3-4), 261--301, 1985.
\bibitem{GLZ09} P. Guan, Q. Li, and X. Zhang, \textit{A uniqueness theorem in K\"ahler geometry}, Math. Ann., \textbf{345}(2), 377--393, 2009. 
\bibitem{GM03} P. Guan and X. Ma, \textit{The Christoffel-Minkowski problem I: convexity of solutions of a Hessian equation}, Invent. Math., \textit{151}(3), 553--577, 2003.
\bibitem{Lu24} S. Lu, \textit{Interior $C^2$ estimate for Hessian quotient equation in general dimension}, Ann. PDE, \textbf{11}, article number 17, 2025.
\bibitem{S25} M. Sroka, \textit{Remarks on Hessian quotient equations on Riemannian manifolds}, J. Funct. Anal., \textbf{289}(10), Paper No. 111123, 2025.
\end{thebibliography}
\end{document}